\documentclass[11pt]{amsart}

\usepackage[T1]{fontenc}
\usepackage{lmodern}
\usepackage{microtype}
\usepackage{mathtools}
\usepackage{amssymb}
\usepackage{enumitem}
\usepackage[colorlinks=true,linkcolor=blue,citecolor=blue,urlcolor=blue]{hyperref}
\hypersetup{pdftitle={Regularity for Elliptic Equations with Coefficients of Small Mean Oscillation},pdfauthor={Luigi D'Onofrio}}

\newcommand{\R}{\mathbb R}
\newcommand{\fint}{\mathop{\rlap{\raisebox{0.25ex}{\scriptsize--}}\!\int}}
\newcommand{\BMO}{\mathrm{BMO}}
\newcommand{\VMO}{\mathrm{VMO}}
\newcommand{\supp}{\operatorname{supp}}
\newcommand{\dist}{\operatorname{dist}}
\newcommand{\dd}{\,\mathrm d}
\newcommand{\cM}{\mathcal M}

\newcommand{\Sym}{\mathbb S}

\theoremstyle{plain}
\newtheorem{theorem}{Theorem}[section]
\newtheorem{lemma}[theorem]{Lemma}
\newtheorem{proposition}[theorem]{Proposition}
\newtheorem{corollary}[theorem]{Corollary}
\theoremstyle{remark}
\newtheorem{remark}[theorem]{Remark}

\title[Elliptic equations with small mean oscillation]
{Regularity for Elliptic Equations with Coefficients of Small Mean Oscillation}
\author{Luigi D'Onofrio}
\address{Department of Science and Technology, University of Naples ``Parthenope'',
Centro Direzionale di Napoli, Isola C4, 80143 Napoli, Italy}
\email{luigi.donofrio@uniparthenope.it}
\date{}

\begin{document}

\begin{abstract}
We give a detailed proof of interior $W^{2,p}$ regularity for uniformly elliptic equations in nondivergence form
\[
 a^{ij}(x)D_{ij}u+b^i(x)D_i u=f.
\]
The leading matrix is assumed to have sufficiently small mean oscillation on the balls under consideration.  The proof uses a modern real-variable argument: a constant-coefficient harmonic replacement yields a sharp-function estimate for the Hessian, and the Fefferman--Stein and Hardy--Littlewood theorems permit the coefficient error to be absorbed.  A parameter estimate, obtained by Agmon's auxiliary-variable argument, supplies a consistent resolvent on all $L^p$ spaces and makes the subsequent gain of integrability non-circular.  The first-order term is retained throughout and is controlled by the scale-invariant quantity $R^{1-n/q}\|b\|_{L^q(B_R)}$, with $q>\max\{n,p\}$.  As a consequence, coefficients in $\VMO_{\rm loc}$ give the usual local $W^{2,p}$ regularity for every finite $p$.
\end{abstract}

\maketitle

\section{Introduction and main results}
Let $\Omega\subset\R^n$, $n\ge2$, be open.  We consider the linear equation
\begin{equation}\label{eq:main}
 Lu:=a^{ij}(x)D_{ij}u+b^i(x)D_i u=f
 \qquad\text{in }\Omega,
\end{equation}
where repeated indices are summed.  For matrices $M=(m^{ij})$ and $N=(n_{ij})$ we use the Frobenius contraction
\[
 M:N:=\sum_{i,j=1}^n m^{ij}n_{ij}.
\]
Thus $A:D^2u=a^{ij}D_{ij}u$.  The matrix $A=(a^{ij})$ is real, symmetric and uniformly elliptic:
\begin{equation}\label{eq:ellipticity}
 \lambda|\xi|^2\le a^{ij}(x)\xi_i\xi_j\le\Lambda|\xi|^2
 \qquad\text{for a.e. }x\in\Omega,
\end{equation}
with $0<\lambda\le\Lambda<\infty$.

For an integrable function $g$ on a ball $B$, write
\[
 g_B:=\fint_B g(x)\,\dd x,
 \qquad
 [g]_{\BMO(B)}:=\sup_{B'\subset B}\fint_{B'}|g-g_{B'}|\,\dd x,
\]
where the supremum is taken over all balls $B'\subset B$.  For matrix-valued functions we use the Euclidean norm on the space of matrices.  If $U\Subset\Omega$, define
\begin{equation}\label{eq:omegaA}
 \omega_A(r;U):=
 \sup_{x\in U}\ \sup_{\substack{0<\rho\le r\\ B_\rho(x)\Subset\Omega}}
 \fint_{B_\rho(x)}|A-A_{B_\rho(x)}|\,\dd x.
\end{equation}
Thus $A\in\VMO_{\rm loc}(\Omega)$ when $\omega_A(r;U)\to0$ as $r\downarrow0$ for every $U\Subset\Omega$.  For the vector field $b$ we analogously set
\begin{equation}\label{eq:omegab}
 \omega_b(r;U):=
 \sup_{x\in U}\ \sup_{\substack{0<\rho\le r\\ B_\rho(x)\Subset\Omega}}
 \fint_{B_\rho(x)}|b-b_{B_\rho(x)}|\,\dd x.
\end{equation}
Accordingly, $b\in\VMO_{\rm loc}(\Omega;\R^n)$ means that
$\omega_b(r;U)\to0$ as $r\downarrow0$ for every $U\Subset\Omega$.

For the absorption of the first-order term we need a different, full-norm quantity.  For $q>n$ set
\begin{equation}\label{eq:beta}
 \beta_{b,q}(r;U):=
 \sup_{x\in U}\ \sup_{\substack{0<\rho\le r\\ B_\rho(x)\Subset\Omega}}
 \rho^{1-n/q}\|b\|_{L^q(B_\rho(x))}.
\end{equation}
If $b\in L^q_{\rm loc}(\Omega)$, then $\beta_{b,q}(r;U)\to0$ as $r\downarrow0$.  Indeed, this follows from the absolute continuity of the integral, uniformly over sets of sufficiently small measure.  In particular, the conclusion holds for every finite $q$ when $b\in\BMO_{\rm loc}(\Omega)$, by the John--Nirenberg inequality.  The full $L^q$ norm in \eqref{eq:beta} is essential: the BMO seminorm alone does not detect a large constant drift.

The interior estimates of Chiarenza, Frasca and Longo \cite{CFL91,CFL93} are classically proved through singular-integral commutators.  Here we use instead the sharp-function method developed in modern $L^p$ theory; see, for example, \cite{Krylov,SteinHA}.  The lower-order term is treated in the spirit of \cite{Vitanza}.

Our local estimate is as follows.

\begin{theorem}[Local $W^{2,p}$ estimate]\label{thm:local}
Let $1<p<\infty$ and choose $q>\max\{n,p\}$.  There exist
\[
 \delta=\delta(n,p,q,\lambda,\Lambda)>0,
 \qquad C=C(n,p,q,\lambda,\Lambda)>0,
\]
such that the following statement holds.  Let $B_{2R}=B_{2R}(x_0)\Subset\Omega$ and assume
\begin{equation}\label{eq:smallness}
 [A]_{\BMO(B_{2R})}\le\delta,
 \qquad
 (2R)^{1-n/q}\|b\|_{L^q(B_{2R})}\le\delta.
\end{equation}
If $u\in W^{2,p}(B_{2R})$ is a strong solution of \eqref{eq:main}, with $f\in L^p(B_{2R})$, then
\begin{equation}\label{eq:localestimate}
 \|D^2u\|_{L^p(B_R)}
 \le C\left(
 \|f\|_{L^p(B_{2R})}+R^{-2}\|u\|_{L^p(B_{2R})}
 \right).
\end{equation}
\end{theorem}

The coefficient to be absorbed in the drift term is therefore $R^{1-n/q}\|b\|_{L^q}$, which becomes small at small scales because $q>n$.  No negative power of $R$ occurs.

The final regularity statement is the following.

\begin{theorem}[VMO regularity]\label{thm:VMO-intro}
Assume that $A\in\VMO_{\rm loc}(\Omega)$ satisfies \eqref{eq:ellipticity} and that $b\in\VMO_{\rm loc}(\Omega;\R^n)$.  Let $u\in W^{2,2}_{\rm loc}(\Omega)$ be a strong solution of \eqref{eq:main}, with $f\in L^p_{\rm loc}(\Omega)$ for some $1<p<\infty$.  Then $u\in W^{2,p}_{\rm loc}(\Omega)$.
\end{theorem}

The proof of Theorem~\ref{thm:VMO-intro} is given in Section~\ref{sec:bootstrap}.  Its key point is that the target $W^{2,p}$ regularity is not assumed in order to apply the estimate: a consistent resolvent provides a genuine finite gain of integrability.

\section{The John--Nirenberg inequality}
The exponential decay is obtained by a recursive stopping-time construction.  We state the result on a cube; the ball formulation follows by the standard comparison of balls and cubes.

For a cube $Q$, let
\[
 [g]_{\BMO(Q)}:=\sup_{Q'\subset Q}\fint_{Q'}|g-g_{Q'}|\,\dd x,
\]
where $Q'$ ranges over subcubes with sides parallel to those of $Q$.

\begin{theorem}[John--Nirenberg]\label{thm:JN}
There exist dimensional constants $c_1,c_2>0$ such that, for every cube $Q$, every $g\in\BMO(Q)$ and every $t>0$,
\begin{equation}\label{eq:JNtail}
 \bigl|\{x\in Q:|g(x)-g_Q|>t\}\bigr|
 \le c_1|Q|\exp\!\left(-c_2\frac{t}{[g]_{\BMO(Q)}}\right).
\end{equation}
\end{theorem}

\begin{proof}
By homogeneity assume $[g]_{\BMO(Q)}=1$.  The function is temporarily restricted to the fixed cube $Q$: every stopping cube and every parent used below is a dyadic subcube of $Q$.  Equivalently, one may extend $(g-g_Q)\chi_Q$ by zero, provided that the stopping construction is still performed only inside $Q$.

We use the stopping height $2$.  Select the maximal dyadic subcubes
$\{Q_j^{(1)}\}$ of $Q$ for which
\[
 \fint_{Q_j^{(1)}}|g-g_Q|\,\dd x>2.
\]
The cube $Q$ itself is not selected, since
$\fint_Q|g-g_Q|\le1$.  The selected cubes are pairwise disjoint and
\begin{equation}\label{eq:gen1measure}
 2\sum_j|Q_j^{(1)}|
 <\sum_j\int_{Q_j^{(1)}}|g-g_Q|\,\dd x
 \le\int_Q|g-g_Q|\,\dd x
 \le |Q|.
\end{equation}
Thus their union occupies at most one half of $Q$.  Outside that union,
dyadic differentiation gives $|g-g_Q|\le2$ a.e.  If
$\widehat Q_j^{(1)}$ is the dyadic parent of $Q_j^{(1)}$, then
$\widehat Q_j^{(1)}\subset Q$, and maximality gives
\[
 \fint_{\widehat Q_j^{(1)}}|g-g_Q|\,\dd x\le2.
\]
Since $|\widehat Q_j^{(1)}|=2^n|Q_j^{(1)}|$, it follows that
\begin{equation}\label{eq:jump}
 |g_{Q_j^{(1)}}-g_Q|
 \le\fint_{Q_j^{(1)}}|g-g_Q|\,\dd x
 \le 2^{n+1}.
\end{equation}

Now repeat the same construction inside each selected cube $P$, applying it to
$g-g_P$.  Since $[g]_{\BMO(Q)}=1$, the descendants selected inside $P$
occupy at most $|P|/2$.  Let $\mathcal G_k$ be the family of cubes selected
at generation $k$ and let
\[
 E_k:=\bigcup_{P\in\mathcal G_k}P.
\]
Summing the one-half estimate over all parents at the preceding generation
gives recursively
\begin{equation}\label{eq:gendecay}
 |E_k|\le\frac12|E_{k-1}|\le 2^{-k}|Q|.
\end{equation}
This is the precise origin of the factor $2^{-k}$.

If $x\in E_{k-1}\setminus E_k$, there is a nested chain
$Q=P_0\supset P_1\supset\cdots\supset P_{k-1}$ of stopping cubes containing
$x$.  At the last stage, $|g(x)-g_{P_{k-1}}|\le2$ a.e., and every jump of
successive averages is bounded by $2^{n+1}$ as in \eqref{eq:jump}.  Therefore
\begin{equation}\label{eq:pointgen}
 |g(x)-g_Q|
 \le 2+\sum_{j=1}^{k-1}|g_{P_j}-g_{P_{j-1}}|
 \le 2+(k-1)2^{n+1}
 \le C_n k.
\end{equation}
Consequently,
\begin{equation}\label{eq:discreteJN}
 \bigl|\{x\in Q:|g(x)-g_Q|>C_nk\}\bigr|
 \le |E_k|\le2^{-k}|Q|,
 \qquad k=1,2,\ldots .
\end{equation}
For $t\ge C_n$, choose $k=\lfloor t/C_n\rfloor$.  Then
$C_nk\le t$ and
\[
 \bigl|\{|g-g_Q|>t\}\bigr|
 \le2^{-k}|Q|
 \le2|Q|\exp\!\left(-\frac{\log2}{C_n}t\right).
\]
For $0<t<C_n$ the same estimate follows after increasing the dimensional
constant.  Restoring $[g]_{\BMO(Q)}$ by homogeneity proves
\eqref{eq:JNtail}.
\end{proof}

\begin{corollary}\label{cor:JN-Ls}
For every $1\le s<\infty$ there is $C=C(n,s)$ such that
\begin{equation}\label{eq:JN-Ls}
 \left(\fint_Q|g-g_Q|^s\,\dd x\right)^{1/s}
 \le C[g]_{\BMO(Q)}.
\end{equation}
The same estimate holds on balls, with a dimensional change in the constant.
\end{corollary}

\begin{proof}
Integrate the distribution function in \eqref{eq:JNtail}.  This is the standard layer-cake argument.
\end{proof}

\section{Sharp functions and the principal part}\label{sec:principal}
We first isolate the constant-coefficient estimates used below.

\begin{proposition}[Constant-coefficient estimates]\label{prop:constant}
Let $\bar A$ be a constant symmetric matrix satisfying \eqref{eq:ellipticity}.
\begin{enumerate}[label=\textup{(\roman*)}]
\item For every $1<s<\infty$ and $\varphi\in C_c^\infty(\R^n)$,
\begin{equation}\label{eq:globalCZ}
 \|D^2\varphi\|_{L^s(\R^n)}
 \le C\|\bar a^{ij}D_{ij}\varphi\|_{L^s(\R^n)}.
\end{equation}
\item If $z\in W^{2,s}(B_\rho)\cap W_0^{1,s}(B_\rho)$ and
$\bar a^{ij}D_{ij}z=F$ in $B_\rho$, then
\begin{equation}\label{eq:dirichletCZ}
 \|D^2z\|_{L^s(B_\rho)}\le C\|F\|_{L^s(B_\rho)}.
\end{equation}
\item If $\bar a^{ij}D_{ij}h=0$ in $B_{\kappa r}(x_0)$, where $\kappa\ge4$, then
\begin{equation}\label{eq:harmonicdecay}
 \left(\fint_{B_r(x_0)}|D^2h-(D^2h)_{B_r(x_0)}|^s\right)^{1/s}
 \le C\kappa^{-1}
 \left(\fint_{B_{\kappa r}(x_0)}|D^2h|^s\right)^{1/s}.
\end{equation}
\end{enumerate}
The constants depend only on $n,s,\lambda,\Lambda$.
\end{proposition}

\begin{proof}
The global estimate \eqref{eq:globalCZ} is the classical Calder\'on--Zygmund estimate after a linear change of variables.  The zero-boundary estimate \eqref{eq:dirichletCZ} is the corresponding Dirichlet estimate on a ball.  References include \cite[Chapter~9]{GT} and \cite[Chapter~II]{SteinSI}.  For \eqref{eq:harmonicdecay}, scale to $B_\kappa$, apply the interior estimate for $D^3h$ on $B_{\kappa/2}$, and then use Poincar\'e's inequality on $B_1$.
\end{proof}

For a locally integrable function $F$ define
\[
 \cM_sF(x):=\bigl(\cM(|F|^s)(x)\bigr)^{1/s},
 \qquad
 F_s^\#(x):=
 \sup_{B\ni x}\left(\fint_B|F-F_B|^s\right)^{1/s},
\]
where $\cM$ is the Hardy--Littlewood maximal operator.  We use the two classical estimates
\begin{equation}\label{eq:FSmax}
 \|F\|_{L^p}\le C\|F_s^\#\|_{L^p},
 \qquad
 \|\cM_sF\|_{L^p}\le C\|F\|_{L^p},
 \qquad 1<s<p<\infty,
\end{equation}
for $F\in L^p$ with compact support.  The first follows from the Fefferman--Stein theorem, since the ordinary sharp function is bounded by $F_s^\#$; see \cite{FeffermanStein,SteinHA}.

We shall also use a local-to-global extension.

\begin{lemma}[Elliptic BMO extension]\label{lem:extension}
Let $B$ be a ball and let $A:B\to\Sym^n$ satisfy \eqref{eq:ellipticity}.  There exists a symmetric matrix field $\widetilde A$ on $\R^n$ such that
\[
 \widetilde A=A\quad\text{a.e. on }B,
 \qquad
 \lambda I\le\widetilde A\le\Lambda I\quad\text{a.e. on }\R^n,
\]
and
\begin{equation}\label{eq:extensionBMO}
 [\widetilde A]_{\BMO(\R^n)}
 \le C_n[A]_{\BMO(B)}.
\end{equation}
\end{lemma}

\begin{proof}
Jones' extension theorem \cite{Jones} gives a componentwise BMO extension $EA$ with seminorm bounded by $C_n[A]_{\BMO(B)}$.  Regard symmetric matrices as a Euclidean space and let $\Pi$ be the metric projection onto the closed convex set
\[
 \mathcal K_{\lambda,\Lambda}:=
 \{M\in\Sym^n:\lambda I\le M\le\Lambda I\}.
\]
The projection $\Pi$ is $1$-Lipschitz and fixes every matrix in $\mathcal K_{\lambda,\Lambda}$.  Set $\widetilde A=\Pi(EA)$.  For every cube or ball $Q$ and every constant matrix $C$,
\[
 \fint_Q|\widetilde A-(\widetilde A)_Q|
 \le2\fint_Q|\widetilde A-\Pi(C)|
 \le2\fint_Q|EA-C|.
\]
Choosing $C=(EA)_Q$ proves \eqref{eq:extensionBMO}.
\end{proof}

The next lemma is the real-variable core of the proof.

\begin{lemma}[Mean oscillation of the Hessian]\label{lem:oscillation}
Let $1<s<\infty$, let $\mu,\nu>1$ be conjugate exponents, and let $\kappa\ge4$.  Suppose that $A$ is defined on $\R^n$, satisfies \eqref{eq:ellipticity}, and belongs to $\BMO(\R^n)$.  If $v\in C_c^\infty(\R^n)$ and $G=A:D^2v$, then for every ball $B_r(x_0)$,
\begin{align}
 &\left(\fint_{B_r(x_0)}
 |D^2v-(D^2v)_{B_r(x_0)}|^s\right)^{1/s}
 \notag\\
 &\quad\le C\kappa^{-1}
 \left(\fint_{B_{\kappa r}(x_0)}|D^2v|^s\right)^{1/s}
 +C\kappa^{n/s}
 \left(\fint_{B_{\kappa r}(x_0)}|G|^s\right)^{1/s}
 \notag\\
 &\qquad
 +C\kappa^{n/s}[A]_{\BMO(\R^n)}
 \left(\fint_{B_{\kappa r}(x_0)}|D^2v|^{s\mu}\right)^{1/(s\mu)}.
 \label{eq:oscillation}
\end{align}
Here $C=C(n,s,\mu,\lambda,\Lambda)$.
\end{lemma}

\begin{proof}
Write $B^*=B_{\kappa r}(x_0)$ and $\bar A=A_{B^*}$.  Let $h$ solve
\[
 \bar A:D^2h=0\quad\text{in }B^*,
 \qquad h=v\quad\text{on }\partial B^*,
\]
and put $z=v-h$.  Then $z\in W^{2,s}(B^*)\cap W_0^{1,s}(B^*)$ and
\begin{equation}\label{eq:z-equation}
 \bar A:D^2z
 =G+(\bar A-A):D^2v.
\end{equation}
By Proposition~\ref{prop:constant}(ii), H\"older's inequality and Corollary~\ref{cor:JN-Ls},
\begin{align}
 \left(\fint_{B^*}|D^2z|^s\right)^{1/s}
 &\le C\left(\fint_{B^*}|G|^s\right)^{1/s}
 \notag\\
 &\quad+C
 \left(\fint_{B^*}|A-\bar A|^{s\nu}\right)^{1/(s\nu)}
 \left(\fint_{B^*}|D^2v|^{s\mu}\right)^{1/(s\mu)}
 \notag\\
 &\le C\left(\fint_{B^*}|G|^s\right)^{1/s}
 +C[A]_{\BMO(\R^n)}
 \left(\fint_{B^*}|D^2v|^{s\mu}\right)^{1/(s\mu)}.
 \label{eq:z-bound}
\end{align}
Proposition~\ref{prop:constant}(iii), together with $h=v-z$, gives
\[
 \left(\fint_{B_r}|D^2h-(D^2h)_{B_r}|^s\right)^{1/s}
 \le C\kappa^{-1}
 \left(\fint_{B^*}|D^2v|^s\right)^{1/s}
 +C\kappa^{-1}
 \left(\fint_{B^*}|D^2z|^s\right)^{1/s}.
\]
Moreover,
\[
 \left(\fint_{B_r}|D^2z-(D^2z)_{B_r}|^s\right)^{1/s}
 \le2\kappa^{n/s}
 \left(\fint_{B^*}|D^2z|^s\right)^{1/s}.
\]
Combining these inequalities with \eqref{eq:z-bound} proves \eqref{eq:oscillation}.
\end{proof}

\begin{theorem}[Global estimate for small BMO coefficients]\label{thm:global-principal}
Let $1<p<\infty$.  There exist
$\delta_p=\delta_p(n,p,\lambda,\Lambda)>0$ and
$C=C(n,p,\lambda,\Lambda)$ such that, whenever $A$ is defined on $\R^n$, satisfies \eqref{eq:ellipticity}, and
$[A]_{\BMO(\R^n)}\le\delta_p$, one has
\begin{equation}\label{eq:global-principal}
 \|D^2v\|_{L^p(\R^n)}
 \le C\|A:D^2v\|_{L^p(\R^n)}
 \qquad\text{for every }v\in W^{2,p}(\R^n).
\end{equation}
\end{theorem}

\begin{proof}
It is enough to begin with $v\in C_c^\infty(\R^n)$.  Choose $1<s<p$ and then $\mu>1$ so close to $1$ that $s\mu<p$; let $\nu$ be conjugate to $\mu$.  Taking the supremum over balls in Lemma~\ref{lem:oscillation} gives
\begin{equation}\label{eq:sharp-pointwise}
 (D^2v)_s^\#
 \le C\kappa^{-1}\cM_s(D^2v)
 +C\kappa^{n/s}\cM_s(A:D^2v)
 +C\kappa^{n/s}[A]_{\BMO}\cM_{s\mu}(D^2v).
\end{equation}
The Fefferman--Stein and maximal estimates \eqref{eq:FSmax} imply
\begin{align*}
 \|D^2v\|_{L^p}
 &\le C\left(\kappa^{-1}
 +\kappa^{n/s}[A]_{\BMO}\right)\|D^2v\|_{L^p}
 +C\kappa^{n/s}\|A:D^2v\|_{L^p}.
\end{align*}
Choose first $\kappa$ large enough that the coefficient produced by $\kappa^{-1}$ is at most $1/4$, and then choose $\delta_p$ so small that the BMO contribution is at most $1/4$.  Absorption proves \eqref{eq:global-principal} for smooth compactly supported functions.  Density gives the general case.
\end{proof}

For the later regularity upgrade we need a parameter version and solvability.  The proof is included to avoid a circular application of the target $W^{2,p}$ regularity.

\begin{theorem}[Resolvent estimate and consistency]\label{thm:resolvent}
Let $1<p<\infty$.  After possibly decreasing $\delta_p$, assume the hypotheses of Theorem~\ref{thm:global-principal}.  For every $\sigma>0$,
\begin{equation}\label{eq:parameter}
 \sigma\|v\|_{L^p}
 +\sqrt\sigma\,\|Dv\|_{L^p}
 +\|D^2v\|_{L^p}
 \le C\|A:D^2v-\sigma v\|_{L^p}
\end{equation}
for all $v\in W^{2,p}(\R^n)$.  Consequently,
\[
 A:D^2-\sigma:W^{2,p}(\R^n)\longrightarrow L^p(\R^n)
\]
is an isomorphism.  The inverses are consistent: if $f\in L^{p_1}\cap L^{p_2}$ and the smallness threshold is valid for both exponents, then the two resolvent solutions coincide as distributions.
\end{theorem}

\begin{proof}
We first prove \eqref{eq:parameter} by Agmon's auxiliary-variable argument.  Define on $\R^{n+1}$
\[
 \widehat A(x,y):=\begin{pmatrix}A(x)&0\\0&1\end{pmatrix}.
\]
The BMO seminorm of $\widehat A$ in $\R^{n+1}$ is bounded by a dimensional multiple of that of $A$ in $\R^n$.  Indeed, for an $(n+1)$-dimensional ball $\widehat B$ of radius $\rho$, compare the mean oscillation with the constant $A_{B_\rho}$, where $B_\rho$ is its projection onto $\R^n$; the vertical section length is at most $2\rho$, and $|\widehat B|\simeq\rho^{n+1}$.  Thus
\[
 \fint_{\widehat B}|\widehat A-(\widehat A)_{\widehat B}|
 \le C_n\fint_{B_\rho}|A-A_{B_\rho}|.
\]
We therefore choose the threshold in dimension $n$ small enough to apply Theorem~\ref{thm:global-principal} in dimension $n+1$.

Let $\zeta\in C_c^\infty(\R)$ be nonzero and set
$\zeta_m(y)=\zeta(y/m)$.  For $v\in C_c^\infty(\R^n)$ put
\[
 V_m(x,y)=v(x)\zeta_m(y)\cos(\sqrt\sigma\,y).
\]
Applying \eqref{eq:global-principal} in $\R^{n+1}$ to $V_m$, dividing by $m^{1/p}$, and letting $m\to\infty$, the terms containing $\zeta_m'$ and $\zeta_m''$ vanish.  The $xx$, $xy$ and $yy$ derivatives respectively yield
$\|D^2v\|_p$, $\sqrt\sigma\|Dv\|_p$ and $\sigma\|v\|_p$, while
\begin{align*}
 (\widehat A:D^2_{x,y})V_m
 &=\zeta_m\cos(\sqrt\sigma y)
   (A:D^2v-\sigma v)\\
 &\quad+v\bigl(\zeta_m''\cos(\sqrt\sigma y)
 -2\sqrt\sigma\,\zeta_m'\sin(\sqrt\sigma y)\bigr).
\end{align*}
This proves \eqref{eq:parameter}; density removes the smoothness assumption.

For solvability, use the method of continuity with
$A_t=(1-t)I+tA$, $0\le t\le1$.  The matrices $A_t$ are uniformly elliptic with constants
$\lambda_0=\min\{1,\lambda\}$ and $\Lambda_0=\max\{1,\Lambda\}$, and their BMO seminorms satisfy
$[A_t]_{\BMO}\le [A]_{\BMO}$.  At $t=0$ the operator $\Delta-\sigma$ is an isomorphism from $W^{2,p}$ onto $L^p$.  Estimate \eqref{eq:parameter} gives uniform injectivity and a uniform bound for the inverse on its range.  For completeness, the set of parameters for which $A_t:D^2-\sigma$ is onto is also closed: if $t_j\to t$ and
$(A_{t_j}:D^2-\sigma)u_j=f$, the uniform estimate bounds $u_j$ in $W^{2,p}$.  After passage to a weakly convergent subsequence, the $L^\infty$ convergence $A_{t_j}\to A_t$ permits passage to the limit in the equation and yields
$(A_t:D^2-\sigma)u=f$.  Openness follows from a Neumann-series perturbation, since
$(A_t-A_{t_0}):D^2$ has operator norm $O(|t-t_0|)$ from $W^{2,p}$ to $L^p$.  Thus the continuity method reaches $t=1$.

Finally, consistency on intersections follows from the same continuation.  At $t=0$ the Fourier-multiplier resolvent is independent of the exponent.  Suppose consistency is known at $t_0$.  On a sufficiently short interval about $t_0$, write the inverse as the Neumann series obtained by perturbing the inverse at $t_0$.  The perturbation
$(A_t-A_{t_0}):D^2$ is bounded simultaneously on $W^{2,p_1}$ and $W^{2,p_2}$, and the interval may be chosen so that the series converges in both spaces.  Consequently, for data in $L^{p_1}\cap L^{p_2}$ the two series have identical terms and define the same distribution.  Finitely many such intervals cover $[0,1]$, so the two resolvents agree on the intersection.
\end{proof}

\section{The local estimate for the complete equation}\label{sec:local}
We now retain the first-order term throughout.

\begin{proposition}[Compact-support estimate for the principal part]\label{prop:compact-principal}
For every $1<p<\infty$ there exists $\varepsilon_p>0$ such that, if $B$ is a ball and
$[A]_{\BMO(B)}\le\varepsilon_p$, then every
$v\in W^{2,p}(\R^n)$ with $\supp v\Subset B$ satisfies
\begin{equation}\label{eq:compact-principal}
 \|D^2v\|_{L^p(\R^n)}
 \le C\|A:D^2v\|_{L^p(B)}.
\end{equation}
\end{proposition}

\begin{proof}
Extend $A$ by Lemma~\ref{lem:extension}.  If $\varepsilon_p$ is chosen so that the extension satisfies the threshold in Theorem~\ref{thm:global-principal}, then \eqref{eq:global-principal} applies.  Since $v$ is supported in $B$, the extended operator agrees with $A:D^2$ on the support of $D^2v$.
\end{proof}

\begin{lemma}[Scaled Sobolev inequalities]\label{lem:sobolev}
Let $1<p<\infty$, $q>\max\{n,p\}$ and define $s$ by
\begin{equation}\label{eq:srelation}
 \frac1p=\frac1q+\frac1s.
\end{equation}
Then:
\begin{enumerate}[label=\textup{(\roman*)}]
\item if $v\in W^{2,p}(\R^n)$ and $\supp v\Subset B_R$, then
\begin{equation}\label{eq:compact-sob}
 \|Dv\|_{L^s(\R^n)}
 \le C R^{1-n/q}\|D^2v\|_{L^p(\R^n)};
\end{equation}
\item if $R\le r\le2R$ and $u\in W^{1,p}(B_r)$, then
\begin{equation}\label{eq:noncompact-sob}
 \|u\|_{L^s(B_r)}
 \le C\left(
 R^{1-n/q}\|Du\|_{L^p(B_r)}
 +R^{-n/q}\|u\|_{L^p(B_r)}
 \right).
\end{equation}
\end{enumerate}
\end{lemma}

\begin{proof}
Both statements are standard scaled Sobolev inequalities.  If $p<n$, the exponent $s$ lies strictly between $p$ and the Sobolev exponent $np/(n-p)$ because $q>n$.  The cases $p=n$ and $p>n$ follow respectively from the borderline and Morrey embeddings.  Scaling gives the displayed powers of $R$.
\end{proof}

\begin{corollary}[Compact-support estimate for $L$]\label{cor:compact-full}
Let $p,q$ satisfy the hypotheses of Lemma~\ref{lem:sobolev}.  There exists
$\varepsilon=\varepsilon(n,p,q,\lambda,\Lambda)>0$ such that, if
\begin{equation}\label{eq:compact-small}
 [A]_{\BMO(B_R)}\le\varepsilon,
 \qquad
 R^{1-n/q}\|b\|_{L^q(B_R)}\le\varepsilon,
\end{equation}
then every $v\in W^{2,p}(\R^n)$ with $\supp v\Subset B_R$ satisfies
\begin{equation}\label{eq:compact-full}
 \|D^2v\|_{L^p(\R^n)}
 \le C\|Lv\|_{L^p(B_R)}.
\end{equation}
\end{corollary}

\begin{proof}
Proposition~\ref{prop:compact-principal}, H\"older's inequality and \eqref{eq:compact-sob} give
\begin{align*}
 \|D^2v\|_{L^p}
 &\le C\bigl(\|Lv\|_{L^p}+\|b\cdot Dv\|_{L^p}\bigr)\\
 &\le C\|Lv\|_{L^p}
 +CR^{1-n/q}\|b\|_{L^q}\|D^2v\|_{L^p}.
\end{align*}
The last term is absorbed by decreasing $\varepsilon$.
\end{proof}

We also use the interpolation inequality
\begin{equation}\label{eq:interpolation}
 \|Du\|_{L^p(B_r)}
 \le C\left(\rho\|D^2u\|_{L^p(B_r)}
 +\rho^{-1}\|u\|_{L^p(B_r)}\right),
 \qquad 0<\rho\le r,
\end{equation}
which follows by scaling and a bounded extension operator for a ball.

\begin{lemma}[Hole-filling]\label{lem:hole}
Let $\Phi:[R,2R]\to[0,\infty)$ be bounded and suppose that for all
$R\le t<r\le2R$,
\[
 \Phi(t)\le\theta\Phi(r)+A+B(r-t)^{-2},
 \qquad 0<\theta<1.
\]
Then $\Phi(R)\le C_\theta(A+BR^{-2})$.
\end{lemma}

\begin{proof}
Choose $\tau\in(\sqrt\theta,1)$ and set
$r_k=R+R(1-\tau^k)$.  Then $r_0=R$, $r_k\uparrow2R$, and
$r_{k+1}-r_k=R(1-\tau)\tau^k$.  Iterating the hypothesis $m$ times gives
\[
 \Phi(R)\le\theta^m\Phi(r_m)
 +A\sum_{k=0}^{m-1}\theta^k
 +BR^{-2}(1-\tau)^{-2}
 \sum_{k=0}^{m-1}(\theta\tau^{-2})^k.
\]
Since $\theta\tau^{-2}<1$ and $\Phi$ is bounded, letting $m\to\infty$ proves the claim.
\end{proof}

\begin{proof}[Proof of Theorem~\ref{thm:local}]
Let $R\le t<r\le2R$.  Choose
$\eta\in C_c^\infty(B_r)$ such that
\[
 \eta=1\text{ on }B_t,
 \qquad |D\eta|\le C(r-t)^{-1},
 \qquad |D^2\eta|\le C(r-t)^{-2}.
\]
In particular, the cutoff is never required to equal one up to the boundary of its support.  Put $v=\eta u$.  The assumptions on $B_{2R}$ imply those of Corollary~\ref{cor:compact-full} on $B_r$, after a harmless reduction of $\delta$.  Hence
\begin{equation}\label{eq:start-local}
 \|D^2u\|_{L^p(B_t)}
 \le\|D^2v\|_{L^p}
 \le C\|Lv\|_{L^p(B_r)}.
\end{equation}
The complete product formula is
\begin{equation}\label{eq:product}
 Lv=\eta f
 +2a^{ij}D_i\eta D_j u
 +u\,a^{ij}D_{ij}\eta
 +u\,b^iD_i\eta.
\end{equation}
The gradient term in the ensuing estimate comes from this product rule, not from the constant-coefficient Calder\'on--Zygmund estimate.

Set
\[
 \beta=(2R)^{1-n/q}\|b\|_{L^q(B_{2R})}.
\]
By H\"older's inequality and \eqref{eq:noncompact-sob},
\begin{align}
 \|ub\cdot D\eta\|_{L^p(B_r)}
 &\le C(r-t)^{-1}\|b\|_{L^q(B_{2R})}\|u\|_{L^s(B_r)}\notag\\
 &\le C\beta(r-t)^{-1}\|Du\|_{L^p(B_r)}
 +C\beta R^{-1}(r-t)^{-1}\|u\|_{L^p(B_r)}.
 \label{eq:driftcutoff}
\end{align}
Since $r-t\le R$, the last term is at most
$C\beta(r-t)^{-2}\|u\|_{L^p(B_r)}$.  From
\eqref{eq:start-local}--\eqref{eq:driftcutoff} we obtain
\begin{align}
 \|D^2u\|_{L^p(B_t)}
 &\le C\|f\|_{L^p(B_{2R})}
 +C(1+\beta)(r-t)^{-1}\|Du\|_{L^p(B_r)}\notag\\
 &\quad+C(1+\beta)(r-t)^{-2}\|u\|_{L^p(B_{2R})}.
 \label{eq:beforeinterp}
\end{align}
Apply \eqref{eq:interpolation} with
\[
 \rho=\frac{\varepsilon(r-t)}{1+\beta},
 \qquad 0<\varepsilon\le1.
\]
This choice is admissible because $r-t\le r$.  More explicitly, the gradient
term in \eqref{eq:beforeinterp} satisfies
\begin{align*}
 C(1+\beta)(r-t)^{-1}\|Du\|_{L^p(B_r)}
 &\le C(1+\beta)(r-t)^{-1}\rho
       \|D^2u\|_{L^p(B_r)}\\
 &\quad+C(1+\beta)(r-t)^{-1}\rho^{-1}
       \|u\|_{L^p(B_r)}\\
 &=C\varepsilon\|D^2u\|_{L^p(B_r)}\\
 &\quad+C\varepsilon^{-1}(1+\beta)^2(r-t)^{-2}
       \|u\|_{L^p(B_r)}.
\end{align*}
Since $\beta\le\delta$, we may decrease $\delta$ so that $\beta\le1$; the
coefficient of the last term is then bounded independently of $r,t$ and
$R$.  Choose $\varepsilon$ so small that $\theta:=C\varepsilon<1$, and set
$\Phi(s):=\|D^2u\|_{L^p(B_s)}$.  Hence
\[
 \Phi(t)\le\theta\Phi(r)
 +C\|f\|_{L^p(B_{2R})}
 +C(r-t)^{-2}\|u\|_{L^p(B_{2R})}.
\]
Lemma~\ref{lem:hole} now yields \eqref{eq:localestimate}.
\end{proof}

\begin{remark}[Uniform absorption on all smaller balls]\label{rem:uniform-absorption}
The absorption above is uniform with respect to the radius.  More precisely,
assume that for some $R_0>0$
\[
 [A]_{\BMO(B_{2R_0}(x_0))}\le\delta,
 \qquad
 (2R_0)^{1-n/q}\|b\|_{L^q(B_{2R_0}(x_0))}\le\delta.
\]
Then for every $0<R\le R_0$, nestedness of the balls and positivity of
$1-n/q$ give
\[
 [A]_{\BMO(B_{2R}(x_0))}
 \le [A]_{\BMO(B_{2R_0}(x_0))}\le\delta
\]
and
\[
 (2R)^{1-n/q}\|b\|_{L^q(B_{2R}(x_0))}
 \le (2R_0)^{1-n/q}\|b\|_{L^q(B_{2R_0}(x_0))}\le\delta.
\]
Consequently the constants in Corollary~\ref{cor:compact-full}, the
interpolation step and the hole-filling argument are independent of
$R\in(0,R_0]$.  In particular, no condition involving a negative power of
$R$ is used.
\end{remark}

\section{Interior estimates and the non-circular bootstrap}\label{sec:bootstrap}
\begin{theorem}[Interior estimate]\label{thm:interior}
Let $1<p<\infty$, choose $q>\max\{n,p\}$, and let
$\Omega'\Subset\Omega''\Subset\Omega$.  Assume \eqref{eq:ellipticity},
$A\in\VMO_{\rm loc}(\Omega)$ and $b\in L^q_{\rm loc}(\Omega)$.  If
$u\in W^{2,p}_{\rm loc}(\Omega)$ is a strong solution of \eqref{eq:main}, with $f\in L^p_{\rm loc}(\Omega)$, then
\begin{equation}\label{eq:interior}
 \|D^2u\|_{L^p(\Omega')}
 \le C\left(
 \|f\|_{L^p(\Omega'')}+\|u\|_{L^p(\Omega'')}
 \right).
\end{equation}
The constant depends on $n,p,q,\lambda,\Lambda$,
$\dist(\Omega',\partial\Omega'')$, the VMO modulus of $A$ on $\Omega''$, and $\|b\|_{L^q(\Omega'')}$.
\end{theorem}

\begin{proof}
Choose $R_0>0$ so small that the two conditions in \eqref{eq:smallness} hold on every ball $B_{2R_0}(x)\Subset\Omega''$ with $x\in\Omega'$.  For $A$ this follows from VMO.  For $b$ it follows from $q>n$ and the absolute continuity of the $L^q$ integral.  By Remark~\ref{rem:uniform-absorption}, for every $x\in\Omega'$ and every $0<R\le R_0$ one then has
\[
 [A]_{\BMO(B_{2R}(x))}\le\delta,
 \qquad
 (2R)^{1-n/q}\|b\|_{L^q(B_{2R}(x))}\le\delta.
\]
Thus the absorption is valid uniformly on all smaller balls, with constants independent of $R$.  Also take $R_0$ smaller than a fixed multiple of
$\dist(\Omega',\partial\Omega'')$.  Cover $\Omega'$ by finitely many balls $B_{R_0}(x_j)$ with bounded overlap and apply Theorem~\ref{thm:local} directly to $u$ on $B_{2R_0}(x_j)$.  No partition of unity is inserted into the equation, so no coefficient term is lost.
\end{proof}

We next convert the a priori estimate into a genuine gain of regularity.

\begin{lemma}[Regularity lifting for the principal operator]\label{lem:lifting}
Let $1<r_0<r<\infty$, let $B$ be a ball, and assume that
$[A]_{\BMO(B)}$ is below the extension and resolvent thresholds corresponding to both $r_0$ and $r$.  Suppose
$v\in W^{2,r_0}(\R^n)$, $\supp v\Subset B$, and
\begin{equation}\label{eq:liftinghyp}
 v\in L^r(\R^n),
 \qquad A:D^2v\in L^r(B).
\end{equation}
Then $v\in W^{2,r}(\R^n)$ and
\begin{equation}\label{eq:liftingest}
 \|v\|_{W^{2,r}(\R^n)}
 \le C\left(\|A:D^2v\|_{L^r(B)}+\|v\|_{L^r}\right).
\end{equation}
\end{lemma}

\begin{proof}
Extend $A$ to $\widetilde A$ by Lemma~\ref{lem:extension}.  Put
$F=\widetilde A:D^2v-v$.  By \eqref{eq:liftinghyp}, $F\in L^r$; since it has bounded support and $r>r_0$, it also belongs to $L^{r_0}$.  The function $v$ is the $W^{2,r_0}$ resolvent solution of
\[
 (\widetilde A:D^2-1)v=F.
\]
Theorem~\ref{thm:resolvent} supplies a $W^{2,r}$ solution for the same datum.  Consistency of the two resolvents identifies that solution with $v$.  Estimate \eqref{eq:liftingest} follows from \eqref{eq:parameter}.
\end{proof}

\begin{corollary}[Finite local bootstrap]\label{cor:bootstrap}
Let $1<p_0<p<\infty$ and choose $q>\max\{n,p\}$.  There exists a finite collection of smallness thresholds, depending only on
$n,p_0,p,q,\lambda,\Lambda$, with the following property.  Let
$B_{2R}\Subset\Omega$, assume that the BMO seminorm of $A$ on $B_{2R}$ is below all these thresholds, and let $b\in L^q(B_{2R})$.  If
$u\in W^{2,p_0}(B_{2R})$ is a strong solution of \eqref{eq:main} and
$f\in L^p(B_{2R})$, then
\[
 u\in W^{2,p}(B_R).
\]
\end{corollary}

\begin{proof}
Set $r_0=p_0$.  While $r_k<n$, define
\begin{equation}\label{eq:recurrence}
 \frac1{r_{k+1}}
 =\max\left\{
 \frac1p,\frac1{r_k}-\frac1n+\frac1q
 \right\}.
\end{equation}
Because $1/n-1/q>0$, the sequence strictly increases and reaches either $p$ or an exponent at least $n$ after finitely many steps.  If $r_k\ge n$, set $r_{k+1}=p$.  Let $N$ be the number of steps and choose nested concentric balls
\[
 B_R=B_{\rho_N}\Subset B_{\rho_{N-1}}\Subset\cdots
 \Subset B_{\rho_0}=B_{2R}.
\]

Assume inductively that $u\in W^{2,r_k}(B_{\rho_k})$.  Choose
$\eta_k\in C_c^\infty(B_{\rho_k})$ with $\eta_k=1$ on
$B_{\rho_{k+1}}$ and put $v_k=\eta_ku$.  Then
\begin{equation}\label{eq:bootstrap-eq}
 A:D^2v_k
 =\eta_k f-\eta_k b\cdot Du
 +2a^{ij}D_i\eta_kD_j u
 +u\,a^{ij}D_{ij}\eta_k.
\end{equation}
Suppose first that $r_k<n$.  Sobolev's theorem gives
$Du\in L^{r_k^*}$, where $1/r_k^*=1/r_k-1/n$.  Hence
\[
 b\cdot Du\in L^t,
 \qquad \frac1t=\frac1q+\frac1{r_k^*}
 =\frac1{r_k}-\frac1n+\frac1q.
\]
By \eqref{eq:recurrence}, $r_{k+1}\le t$ and $r_{k+1}\le p$, so every term on the right of \eqref{eq:bootstrap-eq} belongs to $L^{r_{k+1}}$.  The cutoff term containing $Du$ has even the exponent $r_k^*>t$.  The term containing $u$ is also admissible: if $r_k<n/2$, then the second-order Sobolev exponent satisfies
\[
 \frac1{r_k^{**}}=\frac1{r_k}-\frac2n
 <\frac1{r_{k+1}},
\]
whereas for $r_k\ge n/2$ one has $u\in L^s$ for every finite $s$ (and $u\in L^\infty$ when $r_k>n/2$).

If $r_k=n$, choose $s<\infty$ so that $1/p=1/q+1/s$.  The borderline embedding gives $Du\in L^s$ locally, and hence $b\cdot Du\in L^p$.  If $r_k>n$, Morrey's embedding gives $Du\in L^\infty$, so again $b\cdot Du\in L^p$ because $q>p$.

Thus, in every case,
$A:D^2v_k\in L^{r_{k+1}}$ and $v_k\in L^{r_{k+1}}$.  Lemma~\ref{lem:lifting} gives
$v_k\in W^{2,r_{k+1}}(\R^n)$.  Since $v_k=u$ on
$B_{\rho_{k+1}}$, the induction closes.  After finitely many steps one reaches $r_N=p$.
\end{proof}

\begin{proof}[Proof of Theorem~\ref{thm:VMO-intro}]
If $p\le2$, the conclusion follows locally from the inclusion
$W^{2,2}\subset W^{2,p}$ on sets of finite measure.  Assume $p>2$ and fix
$\Omega'\Subset\Omega''\Subset\Omega$.  Choose $q>\max\{n,p\}$.  By the John--Nirenberg inequality, $b\in L^q(\Omega'')$.  The finite bootstrap uses only finitely many BMO thresholds.  Since $A\in\VMO_{\rm loc}$, there is $R_0>0$ such that all of them hold on every ball of radius at most $2R_0$ contained in $\Omega''$.  Cover $\Omega'$ by finitely many such balls and apply Corollary~\ref{cor:bootstrap} with $p_0=2$.  This gives $u\in W^{2,p}(\Omega')$.  Since $\Omega'$ was arbitrary, the theorem follows.
\end{proof}

\begin{remark}
The assumption $b\in\VMO_{\rm loc}$ in Theorem~\ref{thm:VMO-intro} can be weakened to
$b\in L^q_{\rm loc}$ for one exponent $q>\max\{n,p\}$.  The VMO formulation is retained because it is closest to the original equation and automatically supplies every finite local $L^q$ exponent.
\end{remark}

\end{document}